\documentclass[a4paper,12pt]{article}
\usepackage{amsmath,amsthm, amssymb}
\usepackage{graphicx}
\usepackage{graphics}
\usepackage{color}

\newcommand{\dist}{{\rm dist}}

\newcommand{\Addresses}{{
  \bigskip
  \footnotesize

 \textsc{Department of Mathematics, Saarland University, P.O. Box 151150 \\
66041 Saarbr{\"u}cken, Germany}
and
  \textsc{ Peoples’ Friendship University of Russia (RUDN University),
6 Miklukho-Maklaya St, Moscow, 117198, Russian Federation \\}\par\nopagebreak
  \textit{E-mail address:} \texttt{darya@math.uni-sb.de}

  \bigskip

\textsc{St. Petersburg Department of Steklov Institute, Fontanka 27, St. Petersburg 191023, Russian Federation} and
  \textsc{Faculty of Mathematics and Mechanics, St. Petersburg State University, Universitetskii pr. 28,  St. Petersburg 198504, Russian Federation}\par\nopagebreak
  \textit{E-mail address:} \texttt{al.il.nazarov@gmail.com}
}}

\newtheorem{theorem}{Theorem}

\theoremstyle{definition}
\newtheorem{definition}[theorem]{Definition}

\theoremstyle{definition}
\newtheorem{remark}{Remark}

\theoremstyle{plain}
\newtheorem{thm}{Theorem}[section]
\newtheorem{lemma}[thm]{Lemma}

\author{Darya\,E. Apushkinskaya, Alexander\,I. Nazarov}
\title{On the Boundary Point Principle for divergence-type equations }

\begin{document}
\maketitle

\bibliographystyle{amsalpha}
\newcommand{\tg}{\textup{tg}}
\newcommand{\ep}{\varepsilon}
\newcommand{\osc}[1]{\underset{#1}{\textup{osc}}}

\begin{abstract}
We provide some versions of the Zaremba-Hopf-Oleinik boundary point lemma for general elliptic and parabolic equations in divergence form under the sharp requirements on the coefficients of equations and on the boundaries of domains. 
\end{abstract}

\section{Introduction}
The Boundary Point Principle, known also as the ``normal derivative lemma'', is one of the important tools in qualitative analysis of partial differential equations. This principle states that \textit{a supersolution of a partial differential equation with a minimum value at a boundary point, must increase linearly away from its boundary minimum provided the boundary is smooth enough.}

The history of this famous principle begins with a pioneering paper of S.~Zaremba \cite{Z10} where the above assertion was established for the Laplace equation in a three-dimensional domain $\Omega$ satisfying an interior touching ball condition. Notice that the major part of  all known results on the normal derivative lemma concerns  equations with nondivergence structure and strong solutions.  A key contribution to the investigation of this problem for elliptic equations was made simultaneously and independently by E.~Hopf \cite{H52} and O.A.~Oleinik \cite{O52} (by this reason, all the statements of such type are often called the Hopf-Oleinik lemma). The corresponding comprehensive historical review can be found in \cite{AN16}. 

The case of the divergence-type elliptic equations  
\begin{equation} \label{numer}
{\cal L}u:=-D_i(a^{ij}(x)D_ju)+b^i(x)D_iu=0
\end{equation} 
%
is less studied. It is well known that the Boundary Point Principle fails for uniformly elliptic equations in divergence form with bounded and even continuous  coefficients $a^{ij}(x)$ (see, for instance, \cite{G59}, \cite[Ch.3]{GT83}, \cite[Ch.2]{PS07} and \cite{Naz12}). Thus, the normal derivative lemma requires more smoothness of the leading coefficients. 

The sharp requirements on the regularity of the boundary of a domain, providing the validity of the Boundary Point Principle for the Laplace equation, were independently and simultaneously formulated in the papers \cite{MV67} and \cite{W67}.

The first result for weak solutions of 
(\ref{numer}) was proved  by R.~Finn and D.~Gilbarg \cite{FGi57}. They considered  a two-dimensional bounded domain with $\mathcal{C}^{1,\alpha}$-regular boundary, the H{\"o}lder continuous leading coefficients and continuous lower order coefficients.  Recently,  in \cite{KoKu18} (see also \cite{SdL15}) the normal derivative lemma was established in $n$-dimensional domains ($n \geqslant 3$) for equations with the lower-order coefficients from the Lebesgue space $L^q$, $q>n$, under the same assumptions on the leading coefficients and on the boundary  as in \cite{FGi57}.

The history of the Boundary Point Principle for parabolic equations
is much shorter then for elliptic ones and begins with the papers of L.~Nirenberg \cite{N53} and A.~Friedman \cite{F58}. For a partial bibliography in the nondivergence case we refer the reader to \cite{Naz12}. 

As for the divergence-type parabolic equations 
\begin{equation} \label{numer-par}
\mathcal{M}u := \partial_t u-D_i(a^{ij}(x;t)D_ju)+b^i(x;t)D_iu=0,
\end{equation}
we do not know such results. However, the normal derivative lemma 
for (\ref{numer-par}) can be extracted from the lower bound estimates of the Green function for the operator ${\cal M}$. These estimates were obtained in \cite{Zha02}, \cite{C06} and \cite{CKP12} under various assumptions on the coefficients of ${\cal M}$ and on the boundary of a domain. In particular, \cite{CKP12} deals with cylindrical domains with $\mathcal{C}^{1,\alpha}$-regular lateral surface, Dini-continuous leading coefficients
and lower-order coefficient from the so-called parabolic Kato class (see Remark \ref{weaker-Kato-parab} below).
\medskip

The goal of our paper is to prove the Boundary Point Principle for  the general divergence-type elliptic and parabolic equations under strongly weakened assumptions close to the necessary ones.

\subsection{Notation and conventions}
Throughout the paper we use the following notation:

\noindent$x=(x',x_n)=(x_1,\dots,x_{n-1},x_n)$ is a point in
${\mathbb R}^n$;

\noindent $(x;t)=(x',x_n;t)=(x_1,\dots,x_n;t)$ is a point in $\mathbb{R}^{n+1}$;

\noindent $\mathbb{R}^n_+=\{x\in \mathbb{R}^n : x_n>0\}$, \qquad $\mathbb{R}^{n+1}_+=
\{(x;t)\in \mathbb{R}^{n+1} : x_n>0\}$;

\noindent $|x|, |x'|$ are the Euclidean norms in corresponding spaces;

\noindent $B_r(x^0)$ is the open ball in $\mathbb{R}^n$ with center
$x^0$ and radius $r$; \quad $B_r=B_r(0)$;

\noindent
$Q_r(x^0;t^0)=B_r(x^0) \times (t^0-r^2; t^0)$; \quad $Q_r=Q_r(0;0)$;

\medskip

\noindent $D_i$ denotes the operator of (weak) differentiation with respect
to  $x_i$; 

\noindent
$D=\left(D',D_n\right)=\left( D_1, \dots,D_{n-1},D_n\right) $; \quad $\partial_t=\dfrac{\partial}{\partial t}$.
\vspace{0.2cm} 

\noindent 
We adopt the convention that the indices $i$ and $j$ run
from $1$ to $n$. We also adopt the convention regarding summation
with respect to repeated indices.\medskip

We use standard notation for the functional spaces. For a bounded domain ${\cal E}\subset\mathbb{R}^{n+1}$ we understand $\mathcal{C}^{1,0}_{x,t} (\overline{\cal E})$ as the space of $u\in\mathcal{C}(\overline{\cal E})$ such that $Du\in\mathcal{C}(\overline{\cal E})$.

\begin{definition}
We say that a function $\sigma : [0,1]\rightarrow \mathbb{R}_+$ belongs to the class~$\mathcal{D}$~if 
\begin{itemize}
\item $\sigma$ is increasing, and $\sigma (0)=0$;
\item $\sigma (\tau)/\tau$ is summable and decreasing.
\end{itemize}
\end{definition} 

\begin{remark} \label{remark-1}
It should be noted that our assumption about the decay of $\sigma(\tau)/\tau$ is not restrictive (see \cite[Remark~1.2]{AN16} for details). 
Moreover, we claim that without loss of generality $\sigma$ can be assumed continuously differentiable  on $(0;1]$. Indeed, for any  function $\sigma \in \mathcal{D}$, one can define
$$
\hat{\sigma} (r):=2\int\limits_{r/2}^r \frac{\sigma (\tau)}{\tau} \,d\tau, \qquad r \in (0;1].
$$
It is easy to see that $\hat{\sigma} \in \mathcal{C}^1(0;1]$. Due to monotonicity properties of $\sigma$ and $\dfrac{\sigma (\tau)}{\tau}$, we have
\begin{align*}
\hat{\sigma}' (r)&=\frac{2}{r} \left(\sigma (r)-\sigma (r/2)\right) \geqslant 0,\\
\left(\frac{\hat{\sigma}(r)}{r} \right)'&=\frac{1}{r} \left[ \frac{2 \sigma (r)}{r} -\frac{\sigma (r/2)}{(r/2)}-\frac{2}{r}\int\limits_{r/2}^r \frac{\sigma (\tau)}{\tau}\,d\tau \right] \leqslant 0,
\end{align*}
and  for all $r\in (0;1]$
\begin{equation} \label{double-sigma}
\sigma (r) \leqslant \hat{\sigma} (r) \leqslant 2\sigma (r/2).
\end{equation}
The second inequality in (\ref{double-sigma}) provides $\hat{\sigma} \in \mathcal{D}$. Finally, the first inequality in (\ref{double-sigma}) allows us to use $\hat{\sigma}$ instead of $\sigma$ in all estimates, and the claim follows.
\end{remark}

For $\sigma \in \mathcal{D}$ we define the function $\mathcal{J}_{\sigma}$ as
$$
\mathcal{J}_{\sigma}(s):=\int\limits_0^s 
\frac{\sigma (\tau)}{\tau}\,d\tau .
$$

\begin{definition}
Let $\cal E$ be a bounded domain in $\mathbb{R}^n$. We say that a function $\zeta : {\cal E} \to \mathbb{R}$  belongs to the class $\mathcal{C}^{0, \mathcal{D}}({\cal E})$,  if
\begin{itemize}
\item 
$
|\zeta(x)-\zeta (y)|\leqslant \sigma (|x-y|)$ for all $x,y \in \overline{{\cal E}}$, and $\sigma$ belongs to the class $\mathcal{D}$.
\end{itemize}
\medskip

Similarly, suppose that ${\cal E}$ is a bounded domain in $\mathbb{R}^{n+1}$. A function
$\zeta : {\cal E} \to \mathbb{R}$ is said to  belong to the class $\mathcal{C}_p^{0, \mathcal{D}}({\cal E})$,  if
\begin{itemize}
\item 
$
|\zeta(x;t)-\zeta (y;s)|\leqslant \sigma (\sqrt{|t-s|+|x-y|^2})$ for all $(x;t), (y;s) \in \overline{\cal E}$, 
and $\sigma$ belongs to the class $\mathcal{D}$.
\end{itemize}
\end{definition}
\medskip

We use the letters $C$ and $N$ (with or without indices) to denote various constants. To indicate that, say, $C$ depends on some parameters, we list them in parentheses: $C(\dots)$.


\section{Elliptic case}
Let $\Omega$ be a bounded domain in ${\mathbb R}^n$ with boundary $\partial\Omega$, and let $d(x)$ denote the distance between $x$ and $\partial\Omega$.

We suppose that $\partial\Omega$ satisfies the \textit{interior }$\mathcal{C}^{1,\mathcal{D}}$\textit{-paraboloid condition.} The latter means that  in a local coordinate system $\partial\Omega$ is given by the equation
$x_n=F(x')$, where $F$ is a  $\mathcal{C}^1$-function such that $F(0)=0$ and the inequality
\begin{equation} \label{boundary}
F(x') \leqslant |x'|\cdot\sigma (|x'|).
\end{equation}
holds true in some neighborhood of the origin. Here $\sigma$ is a $\mathcal{C}^1$-function belonging to the class $\mathcal{D}$ (see Remark \ref{remark-1}).

\medskip
Let an operator ${\cal L}$ be defined by the formula (\ref{numer}).
Assume that the coefficients of $\mathcal{L}$ satisfy the following conditions:  
\begin{equation} \label{a-condition}
\begin{gathered}
\nu {\cal I}_n\le (a^{ij}(x))\le\nu^{-1}{\cal I}_n,\\
a^{ij} \in \mathcal{C}^{0,\mathcal{D}}(\Omega)\qquad  \text{for all}  \quad i,j=1,\dots,n,
\end{gathered}
\end{equation}
 and
\begin{equation} \label{b-condition}
\omega(r):=\sup\limits_{x \in \Omega}\int\limits_{B_{r}(x)\cap\Omega}\frac{|\mathbf{b}(y)|}{|x-y|^{n-1}}\cdot\frac{d(y)}{d(y)+|x-y|}\,dy
\to 0 \qquad \text{as}\quad r \to 0.
\end{equation}
Here   $\nu$ is a positive constant, ${\cal
I}_n$ is identity $(n\times n)$-matrix, while $\mathbf{b}(y)=\left( b^1(y), \dots, b^n(y)\right) $.

\medskip
\begin{remark} \label{weaker-Kato}
Notice that condition (\ref{b-condition}) says that the function $\dfrac{|\mathbf{b}(y)|}{|x-y|^{n-1}}\cdot\dfrac{d(y)}{d(y)+|x-y|}$ is integrable uniformly with respect to $x$. Moreover, in any strict interior subdomain of $\Omega$ condition (\ref{b-condition}) means that $\mathbf{b}$ is an element of the Kato class $K_{n,1}$. (For the definition of the scale of the Kato classes $K_{n,\alpha}$ with $\alpha <n$ the reader is referred to the paper \cite{DH98}). However, in the whole domain $\Omega$ our condition (\ref{b-condition}) is weaker then $\mathbf{b}\in K_{n,1}$.
\end{remark}
\smallskip

The main result of this Section is stated as follows.
\begin{thm} \label{main-theorem}
Let $\Omega$ be a bounded domain in $\mathbb{R}^n$ with the boundary $\partial\Omega$ satisfying the interior $\mathcal{C}^{1,\mathcal{D}}$-paraboloid condition,  let $\mathcal{L}$ be defined by (\ref{numer}), and let assumptions (\ref{a-condition})-(\ref{b-condition}) be fulfilled. 

In addition, assume that a nonconstant function $u\in \mathcal{C}^1(\overline{\Omega})$ satisfies, in the weak sense, the inequality
$$
\mathcal{L}u\geqslant 0 \qquad \text{in}\quad \Omega.
$$
Then, if $u$ attends its minimum at a point $x^0\in \partial\Omega$, we have
$$
\frac{\partial u}{\partial \mathbf{n}}(x^0) <0.
$$
Here $\frac{\partial}{\partial\mathbf{n}}$ is the  derivative with respect to  the exterior normal on $\partial\Omega$.
\end{thm}

\begin{remark}\label{remark-3}
Notice that  without restriction, we may assume that $x^0=0$ and $\partial\Omega$ is locally a paraboloid $x_n=|x'|\cdot \sigma (|x'|)$ with  a smooth function $\sigma \in \mathcal{D}$. Further,
all the assumptions on $a^{ij}$ and $\mathbf{b}$ are invariant under the $\mathcal{C}^{1,\mathcal{D}}$-regular change of variables.  So, we may consider $\partial\Omega$  locally  as a flat boundary $x_n=0$ and assume, without loss of generality, that   $B_R \cap \mathbb{R}^n_+ \subset \Omega$ for some $R>0$.
\end{remark}

Consider for $0<\rho<R/2$ the point $x^{\rho}=(0,\dots,0,\rho)$ and the annulus
$$
A_{\rho}:=\{x: \rho/2 < |x-x^{\rho}|<\rho\} \subset \Omega.
$$
Let $x^*$ be an arbitrary point in $\overline{A}_{\rho}$. Following \cite{FGi57} (see also \cite{SdL15})  we define the auxiliary functions $z$ and $\psi_{x^*}$ as solutions of the problems
\begin{equation} \label{auxiliary-1}
\left\{ 
\begin{aligned}
\mathcal{L}_0z&=0 \quad \text{in}\ A_{\rho},\\
z&=1 \quad \text{on}\ \partial B_{\rho /2}(x^{\rho}),\\
z&=0 \quad \text{on}\ \partial B_{\rho}(x^{\rho}),
\end{aligned}
\right.
\qquad \quad
\left\{ 
\begin{aligned}
\mathcal{L}_0^{x^*}\psi_{x^*}&=0 \quad \text{in}\ A_{\rho},\\
\psi_{x^*}&=1 \quad \text{on}\ \partial B_{\rho /2}(x^{\rho}),\\
\psi_{x^*}&=0 \quad \text{on}\ \partial B_{\rho}(x^{\rho}),
\end{aligned}
\right.
\end{equation}
where the operators $\mathcal{L}_0$ and $\mathcal{L}_0^{x^*}$ are determined by the formulas 
$$
\mathcal{L}_0z:=-D_i(a^{ij}(x)D_jz) \qquad \text{and}\qquad \mathcal{L}_0^{x^*}\psi_{x^*}:=-D_i(a^{ij}(x^*)D_j\psi_{x^*}),
$$ respectively. It is well known that $\psi_{x^*} \in \mathcal{C}^\infty(\overline{A}_{\rho})$, and the existence of (unique) weak solution $z$ follows from the general elliptic theory.

\begin{lemma} \label{estimate-w}
There exists $C_1=C_1(n,\nu, \sigma)>0$ such that the inequality 
\begin{equation} \label{gradient-z-psi}
|Dz(x^*)-D\psi_{x^*} (x^*)| \leqslant C_1\,\frac{\mathcal{J}_{\sigma}(2\rho)}{\rho} 
\end{equation}
holds true for all $\rho \leqslant R/2$. 
\end{lemma}

\begin{proof}
Setting $w^{(1)}=z-\psi_{x^*}$ we observe that $w^{(1)}$ vanishes  on $\partial A_{\rho}$. Hence, $w^{(1)}$ can be represented in $A_{\rho}$ as
$$
w^{(1)}(x)=\int\limits_{A_{\rho}} G^{x^*}_{\rho}(x,y)\mathcal{L}_0^{x^*}w^{(1)}(y)\,dy\stackrel{(\star)}=
\int\limits_{A_{\rho}} G^{x^*}_{\rho}(x,y) \left(\mathcal{L}_0^{x^*}z(y)-\mathcal{L}_0 z(y)\right)\,dy, 
$$
where $G^{x^*}_{\rho}$ stands for the Green function of the operator $\mathcal{L}_0^{x^*}$ in $A_{\rho}$. The equality $(\star)$ follows from the relation $\mathcal{L}_0^{x^*}\psi_{x^*}=\mathcal{L}_0 z=0$, see (\ref{auxiliary-1}).

Applying integration by parts we get another version of the representation formula:
\begin{equation} \label{w-Green}
w^{(1)}(x)=\int\limits_{A_{\rho}}D_{y_i}G_{\rho}^{x^*}(x,y)\left(a^{ij}(x^*)-a^{ij}(y)\right)D_jz(y)\,dy.
\end{equation}
Differentiating both sides of equality (\ref{w-Green}) with respect to $x_k$ we get
\begin{equation} \label{Dw-Green}
\begin{gathered}
D_kw^{(1)}(x^*)=\int\limits_{A_{\rho}}D_{x_k}D_{y_i}G_{\rho}^{x^*}(x^*,y)\left(a^{ij}(x^*)-a^{ij}(y)\right)D_jz(y)\,dy,\\
k=1,\dots,n.
\end{gathered}
\end{equation}

According to Lemma~3.2 \cite{GW82}, $z\in \mathcal{C}^1(\overline{A}_{\rho})$, and the following estimate holds for $y\in \overline{A}_{\rho}$: 
\begin{equation}
\label{3.11-GW82}
|Dz(y)|\leqslant \frac{N_1}{\rho},
\end{equation}
where $N_1$ depends only on  $n$, $\nu$, and $\sigma$. Moreover, due to Theorem~3.3 \cite{GW82} we have also the estimate for the Green function $G_{\rho}^{x^*}(x,y)$:
\begin{equation} \label{DGreen}
|D_xD_yG_{\rho}^{x^*}(x,y)|\leqslant \frac{N_2}{|x-y|^n} \qquad x,y \in A_{\rho},
\end{equation}
where $N_2$ is completely determined by  $n$, $\nu$, and $\sigma$.

Finally, combination of (\ref{Dw-Green})-(\ref{DGreen}) with condition (\ref{a-condition}) implies 
$$
|Dw^{(1)}(x^*)|\leqslant \frac{N_1N_2}{\rho} \int\limits_{B_{2\rho}(x^*)}\frac {\sigma (|x^*-y|)}{|x^*-y|^n}\,dy,
$$
and (\ref{gradient-z-psi}) follows.
\end{proof}

Further, we introduce the barrier function $v$ defined as the weak solution of the Dirichlet problem
\begin{equation} \label{2.1-SdL}
\left\{
\begin{aligned}
\mathcal{L}v&=0 \quad \text{in}\quad A_{\rho},\\
v&=1 \quad \text{on}\quad \partial B_{\rho/2}(x^{\rho}),\\
v&=0 \quad \text{on}\quad \partial B_{\rho}(x^{\rho}).
\end{aligned}
\right.
\end{equation}

\begin{thm} \label{Lemma-3.1-SdL}
There exists $\rho_0>0$ such that for all $\rho \leqslant \rho_0$ the problem (\ref{2.1-SdL}) admits a unique solution $v\in \mathcal{C}^1(\overline{A}_{\rho})$. Moreover, the inequality
\begin{equation} \label{3.11a-GW82}
|Dv(x)-Dz(x)|\leqslant C_2\,\frac{\omega (2\rho)}{\rho}
\end{equation}
holds true for any $x\in A_{\rho}$. Here $C_2=C_2(n, \nu, \sigma)>0$, $\rho_0$ is completely defined by  $n$, $\nu$, $\sigma$, and $\omega$, while $z\in \mathcal{C}^1(\overline{A}_{\rho})$ is defined in (\ref{auxiliary-1}).
\end{thm}

\begin{proof}
Consider in $A_{\rho}$ the auxiliary function
$w^{(2)}=v-z$.  We observe that $w^{(2)}$ vanishes  on $\partial A_{\rho}$, and 
$$
\mathcal{L}_0w^{(2)}=\mathcal{L}_0v=\mathcal{L}v-b^iD_iv=-b^i \left(D_iw^{(2)}+D_iz\right) \quad \text{in}\quad A_{\rho}.
$$
Hence, $w^{(2)}$ can be represented in $A_{\rho}$ via corresponding Green function $G_{0,\rho}(x,y)$ as
$$
w^{(2)}(x)=-\int\limits_{A_{\rho}} G_{0,\rho}(x,y)b^i(y)  \left(D_iw^{(2)}(y)+D_iz(y)\right)\,dy.
$$

Differentiation with respect to $x_k$  gives
$$
D_kw^{(2)}(x)=-\int\limits_{A_{\rho}}D_{x_k}G_{0,\rho}(x,y)  b^i(y)\left(D_iw^{(2)}(y)+D_iz(y)\right)\,dy.
$$
Therefore, we get the relation
\begin{equation} \label{operator-T}
\left(\mathbb{I} +\mathbb{T}_1 \right)Dw^{(2)} =-\mathbb{T}_1 Dz, \end{equation}
where $\mathbb{I}$ stands for the identity operator, while $\mathbb{T}_1$ denotes the matrix operator whose $(k,i)$ entries are  integral operators with kernels $D_{x_k}G_{0,\rho}(x,y)b^i(y)$.

The statement of Theorem follows from the next assertion.

\begin{lemma} \label{Lemma-A1}
The operator $\mathbb{T}_1$ is bounded in ${\cal C}(\overline{A}_\rho)$, and
$$
\|\mathbb{T}_1\|_{{\cal C}\to{\cal C}}\leqslant C_3\,\omega(2\rho),
$$
where $C_3$ depends only on $n$, $\nu$, and $\sigma$.
\end{lemma}

\begin{proof} 
Theorem~3.3 \cite{GW82} provides the estimate
\begin{equation}\label{th-3.3-GW82}
|D_xG_{0,\rho}(x,y)| \leqslant N_3 \min \left\{|x-y|^{1-n}; \dist \{y, \partial A_{\rho}\} |x-y|^{-n}\right\}
\end{equation}
for any $x,y \in A_{\rho}$. Here $N_3$ is the constant depending only on $n$, $\nu$, and $\sigma$. 

Since $\dist \{y, \partial A_{\rho}\} \leqslant d(y)$ for any $y\in A_{\rho}$, the combination of estimate (\ref{th-3.3-GW82}) with condition (\ref{b-condition}) gives
\begin{equation}\label{Br(x)}
\int\limits_{B_{r}(x)\cap A_{\rho}}|D_xG_{0,\rho}(x,y)|\,|\mathbf{b}(y)|\,dy\leqslant 2N_3\omega(r), \qquad x\in \overline A_{\rho}, \quad r\leqslant 2\rho.
\end{equation}

For arbitrary vector function ${\bf f}\in {\cal C}(\overline{A}_\rho)$ we have
$$
|\mathbb{T}_1{\bf f}(x)|\leqslant \|{\bf f}\|_{{\cal C}(\overline{A}_\rho)}\cdot \int\limits_{A_{\rho}}|D_xG_{0,\rho}(x,y)|\,|\mathbf{b}(y)|\,dy\leqslant 2N_3\,\omega(2\rho) \cdot \|{\bf f}\|_{{\cal C}(\overline{A}_\rho)}, \ \ x\in \overline{A}_\rho.
$$

It remains to show that $\mathbb{T}_1{\bf f}\in {\cal C}(\overline{A}_\rho)$. For $x,\tilde{x}\in\overline{A}_\rho$ and any small $\delta>0$ we have
$$
\aligned
&(\mathbb{T}_1{\bf f})(x)-(\mathbb{T}_1{\bf f}) (\tilde{x}) =J_1+J_2\\
&:= \Big(\int\limits_{A_{\rho}\cap B_\delta(\tilde{x})}+\int\limits_{A_{\rho}\setminus B_\delta(\tilde{x})}\Big)\big(D_xG_{0,\rho}(x,y)-D_xG_{0,\rho}(\tilde{x},y)\big) \, \left[\mathbf{b}(y)\cdot {\bf f}(y)\right]\,dy.
\endaligned
$$
If $|x-\tilde{x}|\leqslant\delta/2$ then (\ref{Br(x)}) gives
$$
\aligned
|J_1|&\leqslant \|{\bf f}\|_{{\cal C}(\overline{A}_\rho)}\cdot \int\limits_{B_\delta(\tilde{x})\cap A_{\rho}}\big(|D_xG_{0,\rho}(x,y)|+|D_xG_{0,\rho}(\tilde{x},y)|\big)\,|\mathbf{b}(y)|\,dy\\
&\leqslant 2 N_3\,\omega(3\delta/2) \cdot \|{\bf f}\|_{{\cal C}(\overline{A}_\rho)}.
\endaligned
$$
Thus, given $\varepsilon$ we can choose $\delta$ such that $|J_1|\leqslant\varepsilon$.

On the other hand, $D_xG_{0,\rho}(x,y)$ is continuous for $x\ne y$. Thus, it is equicontinuous on the compact set 
$$
\{(x,y):\ x\in \overline{B}_{\delta/2}(\tilde{x})\cap \overline{A}_\rho, \ y\in \overline{A}_\rho\setminus B_\delta(\tilde{x}) \}.
$$
Therefore, for chosen $\delta$ we obtain, as $|x-\tilde{x}|\to 0$,
$$
|J_2|\leqslant \|{\bf f}\|_{{\cal C}(\overline{A}_\rho)}\cdot\int\limits_{A_{\rho}} |\mathbf{b}(y)|\,dy \cdot
\max\limits_{y\in \overline{A}_\rho\setminus B_\delta(\tilde{x})} |D_xG_{0,\rho}(x,y)-D_xG_{0,\rho}(\tilde{x},y)| \to 0,
$$
and the Lemma follows.
\end{proof}

We continue the proof of Theorem \ref{Lemma-3.1-SdL}. Choose the value of $\rho_0$ so small that $\omega (2\rho_0) \leqslant \left(2C_3\right)^{-1}$, where $C_3$ is the constant from Lemma \ref{Lemma-A1}. Then by the Banach theorem the operator $\mathbb{I} +\mathbb{T}_1$ in
(\ref{operator-T}) is invertible. This gives the
existence and uniqueness of $w^{(2)}\in{\cal C}^1(\overline{A}_\rho)$, and thus, the
unique solvability of the problem (\ref{2.1-SdL}). Moreover, Lemma \ref{Lemma-A1} and inequality (\ref{3.11-GW82}) provide (\ref{3.11a-GW82}). The proof is complete.
\end{proof}

To prove Theorem \ref{main-theorem} we need the following maximum principle.

\begin{lemma} \label{Lemma-A2}
Let $\mathcal{L}$ be defined by (\ref{numer}), and let assumptions (\ref{a-condition})-(\ref{b-condition}) be satisfied in a domain $\cal E$. Suppose that a function $w\in {\cal C}^1({\cal E})$ satisfies $\mathcal{L}w \geqslant 0$ in ${\cal E}$.
If $w$ attains its minimum in an interior point of $\cal E$ then $w=const$.
\end{lemma}

\begin{proof} In the paper \cite{Zh96} the Harnack inequality was established for the divergence-type operators  with the H{\"o}lder continuous coefficients $a^{ij}$ and $b^i$ belonging to the Kato class $K_{n,1}$. However, it is mentioned in \cite{Zh96} that the assumption of the H{\"o}lder continuity of leading coefficients is needed only for the pointwise gradient estimate of the Green function for the corresponding parabolic operator ${\cal M}_0$ without lower order coefficients, see (\ref{DGamma0}) below. Since by Theorem~2.6 \cite{CKP12} this estimate holds for operators with Dini coefficients, and $\mathbf{b}\in K_{n,1}$ in any strict interior subdomain of $\Omega$ (see Remark~\ref{weaker-Kato}),
the strong maximum principle holds for the operator $\mathcal{L}$ in $\Omega$.
\end{proof}

\begin{proof} [Proof of Theorem~\ref{main-theorem}]
It is well known that the Boundary Point Principle holds true  for the operator with constant coefficients. 
Using this statement for the operator $\mathcal{L}_0^{x^*}$ (see (\ref{auxiliary-1})) with $x^*=0$ in the annulus $A_1$ and rescaling $A_1$ into $A_{\rho}$  we get the estimate
$$
D_n\psi_0(0)\geqslant \frac{N_4(n, \nu)}{\rho}>0.
$$

Furthermore, the inequalities (\ref{gradient-z-psi}) and (\ref{3.11a-GW82}) imply for sufficiently small $\rho$
$$
\begin{aligned}
D_nv(0) &\geqslant D_n\psi_0(0)-|Dz(0)-D\psi_0(0)|-|Dv(0)-Dz(0)|\\ &\geqslant
\frac{N_4}{\rho}-C_1\,\frac{\mathcal{J}_{\sigma}(2\rho)}{\rho}-C_2\,\frac{\omega (2\rho)}{\rho}\geqslant \frac{N_4}{2\rho}. 
\end{aligned}
$$
We fix such a $\rho$. Since $u$ is nonconstant, Lemma \ref{Lemma-A2} ensures $u-u(0)>0$ on $\partial B_{\rho/2}(x^{\rho})$. Therefore,
we have for sufficiently small $\varepsilon$
$$
{\cal L}(u-u(0)-\varepsilon v)\geqslant 0\quad \text{in} \ \ A_{\rho};\qquad u-u(0)-\varepsilon v\geqslant 0 \quad \text{on} \ \ \partial A_{\rho}.
$$
By Lemma~\ref{Lemma-A2} the estimate $u-u(0) \geqslant \varepsilon v$ holds true in $A_{\rho}$, with equality at the origin. This gives
$$
\frac{\partial u}{\partial \mathbf{n}}(0)=-D_nu(0)\leqslant -\varepsilon D_nv(0),
$$
which completes the proof.
\end{proof}

\begin{remark} \label{remark-4}
Notice that the statement of Theorem~\ref{main-theorem} is also valid for weak supersolutions of equation (\ref{numer}). Namely, let $\Omega$ and the coefficients of $\mathcal{L}$ be the same as in Theorem~\ref{main-theorem}, and let a nonconstant function $u \in W^1_2(\Omega)$ with $|\mathbf{b}\cdot Du|\in L^1(\Omega)$ satisfy in $\Omega$ the inequality $\mathcal{L}u \geqslant 0$ in the weak sense. Then, if $u$ attends its minimum at a point $x^0\in \partial\Omega$, we have
$$
\liminf\limits_{\varepsilon \to 0}\frac{u(x^0-\varepsilon {\bf n}(x^0))-u(x^0)}{\varepsilon} >0.
$$
\end{remark}

\begin{remark} \label{remark-5}
The assumptions on the lower-order coefficients $b^i$ ($i=1,\dots, n$) can be also weakened. In fact, one can take as coefficients $b^i$ the signed measures, satisfying condition (\ref{b-condition}). Indeed, all our arguments require a convergence of the corresponding integrals only.
\end{remark}

\section{Parabolic case}
Let $Q$ be a bounded domain in $\mathbb{R}^{n+1}$ with topological boundary $\partial Q$. We define the parabolic boundary $\partial' Q$ as the set of all points $(x^0;t^0) \in \partial Q$ such that for any $\varepsilon >0$, 
we have $Q_{\varepsilon} (x^0;t^0)\setminus \overline Q\ne\emptyset$.
By $d_p(x;t)$ we denote the parabolic distance between $(x;t)$ and $\partial' Q$  which is defined as follows:
$$
d_p(x;t):=\sup\{\rho >0 : Q_{\rho} (x; t) \cap \partial' Q = \emptyset \}.
$$
Next, we define the lateral surface $\partial'' Q$ as the set of all points $(x^0;t^0) \in \partial' Q$ such that $Q_{\varepsilon} (x^0;t^0)\cap Q\ne\emptyset$ for any $\varepsilon >0$. 

We suppose that $Q$ satisfies the \textit{parabolic interior} $\mathcal{C}^{1, \mathcal{D}}$-\textit{paraboloid} condition. It means that in a local coordinate system $\partial '' Q$ is given by the equation $x_n=F(x';t)$, where $F$ is a  $\mathcal{C}^1$-function such that $F(0;0)=0$ and the inequality 
\begin{equation} \label{par-boundary}
F(x';t) \leqslant \sqrt{|x'|^2 -t} \cdot \sigma (\sqrt{|x'|^2 -t} ) \qquad \text{for}\quad t\leqslant 0
\end{equation}
holds in some neighborhood of the origin. Here $\sigma$ is a $\mathcal{C}^1$-function belonging to the class $\mathcal{D}$ (see Remark \ref{remark-1}).

\medskip

Let an operator ${\cal M}$ be defined by the formula (\ref{numer-par}).
Suppose that the coefficients of $\mathcal{M}$ satisfy the following conditions:  
\begin{equation} \label{a-par-condition}
\begin{gathered}
\nu {\cal I}_n\le (a^{ij}(x;t))\le\nu^{-1}{\cal I}_n,\\
a^{ij} \in \mathcal{C}_p^{0,\mathcal{D}}(Q)\qquad  \text{for all}  \quad i,j=1,\dots,n, 
\end{gathered}
\end{equation}
and 
\begin{equation} \label{b-par-condition}
\omega^-_p(r) \to 0  \quad \text{and} \quad
\omega^+_p(r) \to 0  \qquad \text{as}\quad r \to 0,
\end{equation}
where
\begin{align*}
\omega^-_p(r):=\sup\limits_{(x;t) \in Q}\int\limits_{Q_{r}(x;t)\cap Q}&\frac{|\mathbf{b}(y;s)|}{(t-s)^{(n+1)/2}}\cdot \exp{\left(-\gamma\,\frac{|x-y|^2}{t-s}\right)} \times \\
 \times  &\frac{ d_p(y;s)}{d_p(y;s)+\sqrt{|x-y|^2+t-s}}\,dy ds;
\end{align*}
\begin{align*}
\omega^+_p(r):=\sup\limits_{(x;t) \in Q}\int\limits_{Q_{r}(x;t+r^2)\cap Q} &\frac{|\mathbf{b}(y;s)|}{(s-t)^{(n+1)/2}}\cdot \exp{\left(-\gamma\,\frac{|x-y|^2}{s-t}\right)} \times \\
 \times  &\frac{ d_p(y;s)}{d_p(y;s)+\sqrt{|x-y|^2+s-t}}\,dy ds.
\end{align*}
Here $\nu$ and $\mathcal{I}_n$ are the same as in Section 2, $\mathbf{b}(y;s)=(b^1(y;s), \dots, b^n(y;s))$, and $\gamma$ is a positive constant to be determined later, depending only on $n$, $\nu$ and on the moduli of continuity of the coefficients $a^{ij}$.

\medskip
\begin{remark} \label{weaker-Kato-parab}
Similarly to the elliptic case, in any strict interior subdomain of $\overline{Q}\setminus\partial'Q$ condition (\ref{b-par-condition}) means that $\mathbf{b}$ is an element of the parabolic Kato class ${\bf K}_n$, see \cite{CKP12}. 
Indeed, in this case (\ref{b-par-condition}) can be rewritten as follows:
\begin{equation} \label{needed}
\sup\limits_{(x;t) \in Q}\int\limits_{(t-r^2,t+r^2)\times B_r(x)}\frac{|\mathbf{b}(y;s)|}{|s-t|^{(n+1)/2}}\cdot \exp{\left(-\gamma\,\frac{|x-y|^2}{|s-t|}\right)} \,dy ds\to0
\end{equation}
as $r\to0$.

This condition differs from Definition~3.1 \cite{CKP12} only in that the integration in \cite{CKP12} is over $(t-r^2,t+r^2)\times \mathbb{R}^n$. However, using the covering of $\mathbb{R}^n\setminus B_r(x)$ by the balls of radius $r/3$ one can check that corresponding suprema converge to zero simultaneously.

In the whole domain $Q$ our condition (\ref{b-par-condition}) is weaker then $\mathbf{b}\in {\bf K}_n$.
\end{remark}
\medskip

To formulate the parabolic counterpart of Theorem \ref{main-theorem} we need the following notion.
\begin{definition}
For a point $(x;t)\in \overline{Q}$ we define its {\it dependence set} as the set of all points $(y;s)\in \overline{Q}$ admitting a vector-valued map $\mathfrak{F} : [0,1] \mapsto \mathbb{R}^{n+1}$ such that the last coordinate function $\mathfrak{F}_{n+1}$ is strictly increasing and
$$
\mathfrak{F}(0)=(y;s); \quad \mathfrak{F}(1)=(x;t); \quad \mathfrak{F} ((0,1))\subset Q.
$$
If $Q$ is a right cylinder with generatrix parallel to the $t$-axis, then for any $(x;t)\in \overline{Q}$ the dependence set is $\overline{Q} \cap \{s <t\}$. 
\end{definition}

\begin{thm} \label{par-main-theorem}
Let $Q$ be a bounded domain in $\mathbb{R}^{n+1}$, let $\partial'' Q$ satisfy the interior parabolic $\mathcal{C}^{1, \mathcal{D}}$-paraboloid condition, let $\mathcal{M}$ be defined by (\ref{numer-par}), and let assumptions (\ref{a-par-condition})-(\ref{b-par-condition}) be satisfied.

In addition, assume that a function $u \in \mathcal{C}^{1,0}_{x,t} (\overline{Q})$ satisfies, in the weak sense, the inequality
$$
\mathcal{M}u \geqslant 0 \qquad \text{in} \quad Q.
$$
Then, if $u$ attends its minimum at a point $(x^0;t^0) \in \partial'' Q$, and $u$ is nonconstant on the dependence set of $(x^0;t^0)$, we have
$$
\frac{\partial u}{\partial \mathbf{n}} (x^0;t^0) <0.
$$
Here $\frac{\partial}{\partial \mathbf{n}}$ denotes the derivative with respect to the spatial exterior normal on $\partial'' Q\cap\{t=t^0\}$.
\end{thm}

\begin{remark}\label{remark-7}
Notice that we do not care of the behavior of $u$ after $t^0$. Thus, without loss of generality we suppose $Q=Q \cap \{ t<t_0\}$. Moreover, similarly to the elliptic case, we may assume that $(x^0; t^0)=(0;0)$, and $\partial''Q$ is locally a paraboloid 
$$
x_n=\mathcal{P}(x';t):=\sqrt{|x'|^2-t}\cdot \sigma (\sqrt{|x'|^2-t}),
$$ 
where $\sigma \in \mathcal{D}$ is smooth.
\end{remark}

Next, we flatten the boundary of the paraboloid by the coordinate transform
\begin{equation} \label{coord-transform}
\tilde{x}'=x'; \quad \tilde{t}=t; \quad \tilde{x}_n=x_n-\mathcal{P}(x';t).
\end{equation}

\begin{lemma} \label{Lemma-invariance}
Assumptions (\ref{a-par-condition}) and (\ref{b-par-condition}) on $a^{ij}$ and $\mathbf{b}$ remain valid under transform (\ref{coord-transform}).
\end{lemma}

\begin{proof}
 It is easy to see that 
 $
 |D'\mathcal{P}| \in \mathcal{C}_p^{0,\, \mathcal{D}} \left(Q_R \cap \mathbb{R}^{n+1}_+\right)
 $
for some $R>0$. (Here, we consider $D'\mathcal{P}$ as a function of $(x;t)$-variables, which is independent on $x_n$). Therefore, the ``new'' coefficients $\tilde{a}^{ij}$ satisfy (\ref{a-par-condition}) in $Q_R \cap \mathbb{R}^{n+1}_+$. 

It is also evident that the transformed ``old'' coefficients $b^i$ satisfy (\ref{b-par-condition}). However, the coordinate change (\ref{coord-transform}) generates an additional term $\tilde{b}_n$ which admits the estimate
$$
|\tilde{b}_n (\tilde{x};t)| \leqslant C |\partial_t \mathcal{P}(x'; t)|= C \left( \frac{\sigma (\sqrt{|x'|^2-t})}{\sqrt{|x'|^2-t}} +
\sigma' (\sqrt{|x'|^2-t}) \right).
$$
Estimating the integral  entering in the definition of $\omega^{\pm}_p$, it suffices to assume that $x'=0$ and $t=0$. This gives 
$$
\omega^-_p(r)\leqslant C \int\limits_{Q_{r}}\frac{\exp{\left(-\gamma\,\frac{|y|^2}{-s}\right)}}{(-s)^{(n+1)/2}}\cdot \left( \frac{\sigma (\sqrt{|y'|^2-s})}{\sqrt{|y'|^2-s}} +
\sigma' (\sqrt{|y'|^2-s}) \right)\, dy ds.
$$
After integration over $y_n$ we make change of variables 
$$
\varrho=\frac{|y'|}{\sqrt{-s}}; \qquad \tau=\sqrt{|y'|^2-s},
$$
and arrive at 
\begin{align*}
\omega^-_p(r)&\leqslant C \int\limits_{0}^{r\sqrt{2}}\int\limits_0^{\infty}
\exp{(-\gamma\varrho^2)} \frac {\varrho^{n-2}}{\sqrt{\varrho^2+1}} \left(\frac{\sigma (\tau)}{\tau} +\sigma' (\tau) \right)
d\varrho d\tau \\
&\leqslant C(n,\gamma) \left( \mathcal{J}_{\sigma} (r\sqrt{2}) + \sigma (r \sqrt{2})\right),
\end{align*}
and the lemma follows.
\end{proof}
 
 Thus, we may consider $\partial'' Q$  locally  as a flat boundary $x_n=0$ and assume, without loss of generality, that $Q_R \cap \mathbb{R}^{n+1}_+ \subset Q$. 
 
 Next, we take  for $0<\rho\leqslant R/2$ the cylinder ${\cal A}_{\rho}=Q_{\rho}(x^{\rho};0)$ (as in the elliptic case,  $x^{\rho}=(0,\dots,0,\rho)$). Define the auxiliary function $\widetilde z$ as the solution of the initial-boundary value problem
\begin{equation} \label{auxiliary-1-par}
\left\{ 
\begin{aligned}
\mathcal{M}_0\widetilde z:=\partial_t \widetilde z-D_i(a^{ij}(x;t)D_j\widetilde z)&=&&0 &&\text{in}&&\ {\cal A}_{\rho},\\
\widetilde z&=&&0 &&\text{on}&&\ \partial'' {\cal A}_{\rho},\\
\widetilde z(x;-\rho^2)&=&& \varphi(\tfrac {x-x^{\rho}}{\rho}) &&\text{for}&&\ x\in B_{\rho}(x^{\rho}),
\end{aligned}
\right.
\end{equation}
where $\varphi$ is a smooth cut-off function such that 
$$
\varphi(x)=1 \quad \text{for}\ \ 
|x|<1/2;
\qquad \varphi(x)=0 \quad \text{for}\ \ 
|x|>3/4.
$$
The existence of (unique) weak solution $\widetilde z$ follows from the general parabolic theory.

\begin{thm}\label{estimate-z-par}
The function $\widetilde z$ belongs to $\mathcal{C}^{1,0}_{x,t}(\overline{\cal A}_{\rho})$ for sufficiently small $\rho$. Moreover, there exists a positive constant $\widetilde \rho_0\leqslant R/2$ depending only on $n$, $\nu$ and $\sigma$, such that the inequality
\begin{equation} \label{gradient-z}
|D\widetilde z(x;t)| \leqslant \frac{C_4(n,\nu)}{\rho},\qquad (x;t)\in \overline{\cal A}_{\rho}, 
\end{equation}
holds true for all $\rho \leqslant \widetilde \rho_0$. 
\end{thm}

\begin{proof} We partially follow the line of proof of Lemma \ref{estimate-w}. Let $(x^*;t^*)$ be an arbitrary point in $\overline{\mathcal{A}}_{\rho}$. We introduce the auxiliary function $\psi_{x^*\!,\,t^*}$ as the solution of the problem
\begin{equation*}
\left\{ 
\begin{aligned}
\mathcal{M}_0^{x^*\!,\,t^*}\psi_{x^*\!,\,t^*} 
&=&&0 
&&\text{in}&&\ {\cal A}_{\rho},\\
\psi_{x^*\!,\,t^*}&=&&0 &&\text{on}&&\ \partial'' {\cal A}_{\rho},\\
\psi_{x^*\!,\,t^*}(x;-\rho^2)&=&& \varphi(\tfrac {x-x^{\rho}}{\rho}) &&\text{for}&&\ x\in B_{\rho}(x^{\rho}),
\end{aligned}
\right.
\end{equation*}
where $\mathcal{M}_0^{x^*\!,\,t^*}\!:=\partial_t 
-D_ia^{ij}(x^*;t^*)D_j$ is operator with constant coefficients frozen at the point $(x^*;t^*)$. It is well known that $\psi_{x^*,t^*} \in \mathcal{C}^\infty(\overline{\cal A}_{\rho})$, and 
\begin{equation}
\label{3.11-GW82-par}
|D\psi_{x^*,t^*}(y;s)|\leqslant \frac{N_5(n,\nu)}{\rho},\qquad (y;s)\in \overline{\mathcal{A}}_{\rho}.
\end{equation}

Setting $w^{(3)}=\widetilde z-\psi_{x^*,t^*}$ we observe that $w^{(3)}$ vanishes  on $\partial' {\mathcal{A}}_{\rho}$. Hence, $w^{(3)}$ can be represented in the cylinder $\mathcal{A}_{\rho}$ as
$$
w^{(3)}(x;t)=\int\limits_{{\mathcal{A}}_{\rho}\cap \{s\leqslant t\}} \Gamma^{x^*\!,\,t^*}_{\rho}(x,y;t,s)\mathcal{M}_0^{x^*\!,\,t^*}w^{(3)}(y;s)\,dyds, 
$$
where $\Gamma^{x^*\!,\,t^*}_{\rho}$ stands for the Green function of the operator $\mathcal{M}_0^{x^*\!,\,t^*}$ in ${\mathcal{A}}_{\rho}$. 

Similarly to (\ref{w-Green}), we integrate by parts and obtain
\begin{equation*}
w^{(3)}(x;t)=\!
\int\limits_{{\mathcal{A}}_{\rho}\cap \{s\leqslant t\}}\!
D_{y_i}\Gamma_{\rho}^{x^*\!,\,t^*}(x,y;t,s)\left(a^{ij}(x^*;t^*)-a^{ij}(y;s)\right)D_j\widetilde z(y;s)\,dyds.
\end{equation*}
Differentiating both sides with respect to $x_k$, $k=1,\dots,n$, we get the system of equations
\begin{multline} \label{Dw-Green-par}
D_k\widetilde z(x;t)-\!\int\limits_{{\cal A}_{\rho}\cap \{s\leqslant t\}}\!D_{x_k}D_{y_i}\Gamma^{x^*\!,\,t^*}_{\rho}(x,y;t,s)\times\\
\times\left(a^{ij}(x^*;t^*)-a^{ij}(y;s)\right)D_j\widetilde z(y;s)\,dyds
=D_k\psi_{x^*\!,\,t^*}(x;t). 
\end{multline}

Now we put $(x^*;t^*)=(x;t)$ and get the relation
\begin{equation} \label{operator-T2}
\left(\mathbb{I} -\mathbb{T}_2 \right)D\widetilde z = {\boldsymbol \Psi}, 
\end{equation}
where
\begin{equation*}
{\boldsymbol \Psi}=D\psi_{x^*\!,\,t^*}(x;t)\big|_{(x^*;t^*)=(x;t)}
\end{equation*}
while $\mathbb{T}_2$ denotes the matrix integral operator whose kernel is matrix $T_2$ with entries
\begin{equation*}
T_2^{kj}(x,y;t,s)=D_{x_k}D_{y_i}\Gamma^{x^*\!,\,t^*}_{\rho}(x,y;t,s)\big|_{(x^*;t^*)=(x,t)}\cdot\left(a^{ij}(x;t)-a^{ij}(y;s)\right)\chi_{\{s\le t\}}.
\end{equation*}
It is easy to see that ${\boldsymbol \Psi}\in{\cal C}(\overline{\cal A}_\rho)$. Therefore, the statement of Theorem follows from the next assertion.

\begin{lemma} \label{Lemma-A3}
The operator $\mathbb{T}_2$ is bounded in ${\cal C}(\overline{\cal A}_\rho)$, and
$$
\|\mathbb{T}_2\|_{{\cal C}\to{\cal C}}\leqslant C_5\,\mathcal{J}_{\sigma} (2\sqrt{2}\rho),
$$
where $C_5$ depends only on $n$ and $\nu$.
\end{lemma}

\begin{proof} 
The following estimate for the Green function $\Gamma^{x^*\!,\,t^*}_{\rho}(x,y;t,s)$ is well known:
\begin{equation} \label{D2Green-par}
|D_xD_y\Gamma^{x^*\!,\,t^*}_{\rho}(x,y;t,s)|\leqslant \frac{N_6}{(t-s)^{(n+2)/2}} \exp\Big(-N_7\,\frac{|x-y|^2}{t-s}\Big),
\end{equation}
where $N_6$ and $N_7$ are completely determined by  $n$ and $\nu$.

Combination of (\ref{D2Green-par}) with condition (\ref{a-par-condition}) gives for $r\leqslant 2\rho$ and $(x;t)\in\overline{\cal A}_\rho$
\begin{multline*}
\int\limits_{Q_r(x;t)\cap{\cal A}_{\rho}}\!|T_2(x,y;t,s)|\,dyds\\
\leqslant \int\limits_{t-r^2}^t\int\limits_{B_r(x)}\frac{N_6\,\sigma (\sqrt{t-s+|x-y|^2})}{(t-s)^{ (n+2)/2}} \exp\Big(-N_7\,\frac{|x-y|^2}{t-s}\Big)\,dyds.
\end{multline*}
Change of variables $\varrho=|x-y|/\sqrt{t-s}$, 
$\tau=\sqrt{t-s+|x-y|^2}$ gives
\begin{equation} \label{Qr}
\begin{aligned}
&\int\limits_{Q_r(x;t)\cap{\cal A}_{\rho}}\!|T_2(x,y;t,s)|\,dyds\\
\leqslant &\int\limits_0^{r\sqrt{2}}\int\limits_0^{\infty}N_8
\exp{(-N_7\varrho^2)} \varrho^{n-1}\frac{\sigma (\tau)}{\tau}\, 
d\varrho d\tau 
\leqslant C_5 \mathcal{J}_{\sigma} (r\sqrt{2})
\end{aligned}
\end{equation}
($N_8$ and $C_5$ depend only on $n$ and $\nu$).

For a vector function ${\bf f}\in {\cal C}(\overline{\cal A}_\rho)$ and for all $(x;t)\in \overline{\cal A}_\rho$ we have
$$
|\mathbb{T}_2{\bf f}(x;t)|\leqslant \|{\bf f}\|_{{\cal C}(\overline{\cal A}_\rho)}\cdot \int\limits_{{\cal A}_{\rho}}|T_2(x,y;t,s)|\,dyds\leqslant C_5\,\mathcal{J}_{\sigma} (2\sqrt{2}\rho) \cdot \|{\bf f}\|_{{\cal C}(\overline{\cal A}_\rho)}.
$$

It remains to show that $\mathbb{T}_2{\bf f}\in {\cal C}(\overline{\cal A}_\rho)$. For $(x;t),(\tilde{x};\tilde{t})\in\overline{\cal A}_\rho$ and any small $\delta>0$ we have
$$
\aligned
&(\mathbb{T}_2{\bf f})(x;t)-(\mathbb{T}_2{\bf f}) (\tilde{x};\tilde{t}) =\widetilde J_1+\widetilde J_2\\
&:= \Big(\int\limits_{{\cal A}_{\rho}\cap Q_\delta(\tilde{x};\tilde{t})}+\int\limits_{{\cal A}_{\rho}\setminus Q_\delta(\tilde{x};\tilde{t})}\Big)\big(T_2(x,y;t,s)-T_2(\tilde{x},y;\tilde{t},s)\big) {\bf f}(y;s)\,dyds.
\endaligned
$$
Similarly to the proof of Theorem \ref{Lemma-3.1-SdL}, if 
$\sqrt{|t-s|+|x-y|^2}\leqslant\delta/2$ then (\ref{Qr}) gives
$$
|\widetilde J_1|\leqslant 2 C_5\,\mathcal{J}_{\sigma} (3\sqrt{2}\rho/2) \cdot \|{\bf f}\|_{{\cal C}(\overline{\cal A}_\rho)}.
$$
Thus, given $\varepsilon$ we can choose $\delta$ such that $|\widetilde J_1|\leqslant\varepsilon$. 

Next, $D_xD_y\Gamma^{x^*\!,\,t^*}_{\rho}(x,y;t,s)$ is continuous w.r.t. $(x;t)$ and w.r.t. $(x^*;t^*)$ for $(x;t)\ne (y;s)$. Therefore, $T_2(x,y;t,s)$ is continuous w.r.t. $(x;t)$ for $(x;t)\ne (y;s)$. Similarly to the proof of Theorem \ref{Lemma-3.1-SdL}, for chosen $\delta$ we obtain, as $(x;t)\to(\tilde{x};\tilde{t})$,
$$
|\widetilde J_2|\leqslant \|{\bf f}\|_{{\cal C}(\overline{\cal A}_\rho)}\cdot \text{meas}(\overline{\cal A}_\rho)\cdot\!
\max\limits_{(y;s)\in {\cal A}_{\rho}\setminus Q_\delta(\tilde{x};\tilde{t})} |T_2(x,y;t,s)-T_2(\tilde{x},y;\tilde{t},s)| \to 0,
$$
and the Lemma follows.
\end{proof}

We continue the proof of Theorem \ref{estimate-z-par}. Choose the value of $\widetilde \rho_0$ so small that $\mathcal{J}_{\sigma} (2\sqrt{2}\,\widetilde \rho_0) \leqslant \left(2C_5\right)^{-1}$, where $C_5$ is the constant from Lemma \ref{Lemma-A3}. Then by the Banach theorem the operator $\mathbb{I} -\mathbb{T}_2$ in
(\ref{operator-T2}) is invertible. This gives $\widetilde z\in \mathcal{C}^{1,0}_{x,t} (\overline{\cal A}_\rho)$. Moreover, Lemma \ref{Lemma-A2} and inequality (\ref{3.11-GW82-par}) provide (\ref{gradient-z}). 
The proof is complete.
\end{proof}

For $\rho \leqslant \widetilde \rho_0$ we introduce the Green function $\Gamma_{0,\rho}(x,y;t,s)$ of the operator $\mathcal{M}_0$ in the cylinder $\mathcal{A}_{\rho}$. By Theorem~2.6 \cite{CKP12}, $D_x\Gamma_{0,\rho}(x,y;t,s)$ is continuous for $(x;t)\ne (y;s)$, and the estimate 
\begin{equation}\label{DGamma0}
\begin{aligned}
|D_x\Gamma_{0,\rho}(x,y;t,s)| &\leqslant N_8 \min \Big\{\frac 1{(t-s)^{(n+1)/2}}; \frac {\dist \{y, \partial B_{\rho}(x^{\rho})\}} {(t-s)^{(n+2)/2}} \Big\} \times \\
&\times \exp\Big(-N_9\,\frac{|x-y|^2}{t-s}\Big)
\end{aligned}
\end{equation}
holds for any $(x;t), (y;s) \in {\cal A}_{\rho}$, $s<t$. Here $N_8$ and $N_9$ are the constants depending only on $n$, $\nu$, and $\sigma$. \medskip

Further, we introduce the barrier function $\widetilde v$ defined as the weak solution of the initial-boundary value problem
\begin{equation} \label{auxiliary-2-par}
\left\{
\begin{aligned}
\mathcal{M}\widetilde v&=&&0 &&\text{in}&&\ {\cal A}_{\rho},\\
\widetilde v&=&&0 &&\text{on}&&\ \partial'' {\cal A}_{\rho},\\
\widetilde v(x;-\rho^2)&=&& \varphi(\tfrac {x-x^{\rho}}{\rho}) &&\text{for}&&\ x\in B_{\rho}(x^{\rho}),
\end{aligned}
\right.
\end{equation}
where $\varphi$ is the same as in (\ref{auxiliary-1-par}).

\begin{thm} \label{estimate-v-par}
Let $\bf b$ satisfy the first relation in (\ref{b-par-condition}) with $\gamma=N_9(n,\nu,\sigma)$ (here $N_9$ is the constant in (\ref{DGamma0})). Then there exists a positive $\widehat\rho_0\leqslant\widetilde \rho_0$ such that for all $\rho \leqslant \widehat\rho_0$ the problem (\ref{auxiliary-2-par}) admits a unique solution $\widetilde v\in \mathcal{C}^{1,0}_{x,t}(\overline{\cal A}_{\rho})$. Moreover, the inequality
\begin{equation} \label{gradient-v-z}
|D\widetilde v(x;t)-D\widetilde z(x;t)|\leqslant C_6\,\frac{\omega^-_p (2\rho)}{\rho}
\end{equation}
holds true for any $(x;t)\in {\cal A}_{\rho}$. Here $C_6=C_6(n, \nu, \sigma)>0$, $\widehat\rho_0$ is completely defined by  $n$, $\nu$, $\sigma$, and $\omega$, while $\widetilde z\in \mathcal{C}^{1,0}_{x,t}(\overline{\cal A}_{\rho})$ is defined in (\ref{auxiliary-1-par}).
\end{thm}

\begin{proof}
We follow the line of proof of Theorem \ref{Lemma-3.1-SdL}.
Consider in ${\cal A}_{\rho}$ the auxiliary function
$w^{(4)}=\widetilde v-\widetilde z$.  We observe that $w^{(4)}$ vanishes  on $\partial' {\cal A}_{\rho}$, and 
$$
\mathcal{M}_0w^{(4)}=-b^i \left(D_iw^{(4)}+D_i\widetilde z\right) \quad \text{in}\quad {\cal A}_{\rho}.
$$
Similarly to the proof of Theorem \ref{Lemma-3.1-SdL}, $D_kw^{(4)}$ can be represented in ${\cal A}_{\rho}$ as
\begin{align*}
D_kw^{(4)}(x;t)= &-\int\limits_{{\cal A}_{\rho}\cap \{s\leqslant t\}} D_{x_k}\Gamma_{0,\rho}(x,y;t,s)\times\\
&\times b^i(y;s)  \left(D_iw^{(4)}(y;s)+D_i\widetilde z(y;s)\right)\,dyds.
\end{align*}
Therefore, we get the relation
\begin{equation} \label{operator-T3}
\left(\mathbb{I} +\mathbb{T}_3 \right)Dw^{(4)} =-\mathbb{T}_3 D\widetilde z, 
\end{equation}
where $\mathbb{T}_3$ denotes the matrix operator whose $(k,i)$ entries are integral operators with kernels $D_{x_k}\Gamma_{0,\rho}(x,y;t,s)b^i(y;s)\chi_{\{s\le t\}}$.

The statement of Theorem follows from the next assertion.

\begin{lemma} \label{Lemma-A4}
The operator $\mathbb{T}_3$ is bounded in ${\cal C}(\overline{\cal A}_\rho)$, and
$$
\|\mathbb{T}_3\|_{{\cal C}\to{\cal C}}\leqslant C_7\,\omega^-_p(2\rho),
$$
where $C_7$ depends only on $n$, $\nu$, and $\sigma$.
\end{lemma}

\begin{proof} Recall that $\rho\le R/2$ and $Q_R \cap \mathbb{R}^{n+1}_+\subset Q$. Thus $\dist \{y, \partial B_{\rho}(x^{\rho})\} \leqslant d_p(y;s)$ for any $(y;s)\in {\cal A}_{\rho}$, and the combination of estimate (\ref{DGamma0}) with the first relation in (\ref{b-par-condition}) gives for $r\leqslant 2\rho$
\begin{equation}\label{Qr2}
\int\limits_{Q_r(x;t)\cap{\cal A}_{\rho}}\!
|D_x\Gamma_{0,\rho}(x,y;t,s)|\,|\mathbf{b}(y;s)|\,dyds
\leqslant N_{10}(n)N_8\,\omega^-_p(r), \quad 
x\in \overline {\cal A}_{\rho},
\end{equation}
(here $N_8$ is the constant in (\ref{DGamma0})).

The rest of the proof repeats literally the proof of Lemma \ref{Lemma-A1}.
\end{proof}

We continue the proof of Theorem \ref{estimate-v-par}. Choose the value of $\widehat \rho_0$ so small that $\omega (2\widehat \rho_0) \leqslant \left(2C_7\right)^{-1}$, where $C_7$ is the constant from Lemma \ref{Lemma-A4}. Then by the Banach theorem the operator $\mathbb{I} +\mathbb{T}_3$ in
(\ref{operator-T3}) is invertible. This gives the existence and uniqueness of $w^{(4)}\in\mathcal{C}^{1,0}_{x,t}(\overline{\cal A}_\rho)$, and thus, the unique solvability of the problem (\ref{auxiliary-2-par}). Moreover, Lemma \ref{Lemma-A4} and inequality (\ref{gradient-z}) provide (\ref{gradient-v-z}). The proof is complete.
\end{proof}

To prove Theorem \ref{par-main-theorem} we need the following maximum principle.

\begin{lemma} \label{Lemma-A5}
Let $\mathcal{M}$ be defined by (\ref{numer-par}), and let assumptions (\ref{a-par-condition})-(\ref{b-par-condition}) be satisfied in a domain ${\cal E} \subset \mathbb{R}^{n+1}$. Let a function $w\in \mathcal{C}^{1,0}_{x,t}({\cal E})$ satisfy $\mathcal{M}w \geqslant 0$ in ${\cal E}$. If $w$ attains its minimum in a point $(x^0;t^0)\in\overline{\cal E}\setminus \partial'{\cal E}$ then $w=const$ on the closure of the dependence set of $(x^0;t^0)$.
\end{lemma}

\begin{proof} The Harnack inequality for parabolic divergence-type operators was established in \cite{Zh96} under the assumptions that the leading coefficients $a^{ij}$ are H{\"o}lder continuous and $\mathbf{b}$ satisfy (\ref{needed}) with arbitrary $\gamma>0$ (and integration over $(t-r^2,t+r^2)\times \mathbb{R}^n$ that is inessential, see Remark \ref{weaker-Kato-parab}). 

As it was mentioned in the proof of Lemma \ref{Lemma-A2}, the first assumption can be replaced by the Dini continuity. Further, in fact only (\ref{needed}) with a certain $\gamma$ occuring in the estimate of $D\Gamma_{0,\rho}$ is used in \cite{Zh96}. The latter coincides with the assumption $\mathbf{b}\in {\bf K}_n$.

Since our assumption (\ref{b-par-condition}) implies $\mathbf{b}\in {\bf K}_n$ in any strict interior subdomain of $\overline{Q} \setminus \partial' Q$ (see Remark~\ref{weaker-Kato-parab}), the strong maximum principle holds for the operator $\mathcal{M}$.
\end{proof}

\begin{remark}\label{remark-8}
This Lemma is the only point where we need the second relation in (\ref{b-par-condition}). If we could prove at least weak maximum principle for the operator $\cal M$ using only the quantity $\omega^-_p$, we did not need $\omega^+_p$ at all. Unfortunately, we cannot do it, and the question whether the second relation in (\ref{b-par-condition}) is necessary for the Boundary Point Principle remains open.
\end{remark}

\begin{proof} [Proof of Theorem~\ref{par-main-theorem}]
It is well known that the Boundary Point Principle holds true  for the operator with constant coefficients. 
Using this statement for the operator $\mathcal{M}_0^{x^*\!,\,t^*}$ with $x^*=0$, $t^*=0$ in the cylinder ${\cal A}_1$ and rescaling ${\cal A}_1$ into ${\cal A}_{\rho}$  we get the estimate
$$
D_n\psi_{0,0}(0;0)\geqslant \frac{N_{11}(n, \nu)}{\rho}.
$$

Next, the relation (\ref{Dw-Green-par}), Lemma~\ref{Lemma-A3}, and inequality (\ref{gradient-z}) imply for sufficiently small $\rho$
$$
D_n\widetilde z(0;0) \geqslant D_n\psi_{0,0}(0;0)-
\|\mathbb{T}_2D\widetilde z\|_{{\cal C}(\overline{\cal A}_\rho)}\geqslant \frac{N_{11}}{\rho}-C_4C_5\frac {\mathcal{J}_{\sigma} (2\sqrt{2}\rho)}{\rho}
\geqslant \frac{N_{11}}{2\rho}.
$$
The relation (\ref{gradient-v-z}) gives for sufficiently small $\rho$
$$
D_n\widetilde v(0;0) \geqslant D_n\widetilde z(0;0)- |D\widetilde v(0;0)-D\widetilde z(0;0)|\geqslant
\frac{N_{11}}{2\rho}-C_6\,\frac{\omega^-_p (2\rho)}{\rho}\geqslant \frac{N_{11}}{4\rho}. 
$$
We fix such a $\rho$. Since $u$ is nonconstant on the dependence set of $(0;0)$, Lemma \ref{Lemma-A5} ensures 
$$
u(x;-\rho^2)-u(0;0)>0 \qquad \text{for}\quad x\in B_{3\rho/4}(x^{\rho}).
$$ 
Therefore, we have for sufficiently small $\varepsilon$
$$
{\cal M}(u-u(0;0)-\varepsilon \widetilde v)\geqslant 0\quad \text{in} \ \ {\cal A}_{\rho};\qquad u-u(0;0)-\varepsilon \widetilde v\geqslant 0 \quad \text{on} \ \ \partial' {\cal A}_{\rho}.
$$
By Lemma~\ref{Lemma-A5} the estimate $u-u(0;0) \geqslant \varepsilon \widetilde v$ holds true in ${\cal A}_{\rho}$, with equality at the origin. This gives
$$
\frac{\partial u}{\partial \mathbf{n}}(0;0)=-D_nu(0;0)\leqslant -\varepsilon D_n\widetilde v(0;0),
$$
which completes the proof.
\end{proof}

\begin{remark} \label{remark-9}
As in elliptic case, the statement of Theorem~\ref{par-main-theorem} is also valid for weak supersolutions of equation (\ref{numer-par}). Namely, let $Q$ and the coefficients of $\mathcal{M}$ be the same as in Theorem~\ref{par-main-theorem}. Suppose that for a  function $u\in L^2(Q)$ with $Du \in L^2(Q)$ the  assumptions
$$
\sup\limits_{t}\|u(\cdot;t)\|_{L^2}<\infty, \quad \text{and} \quad |\mathbf{b}\cdot Du|\in L^1(Q)
$$
are fulfilled. Finally, let $u$ satisfy
the inequality $\mathcal{M}u \geqslant 0$ in the weak sense. 
 
 Then, if $u$ attends its minimum at a point $(x^0;t^0) \in \partial'' Q$, and $u$ is nonconstant on the dependence set of $(x^0;t^0)$, we have
$$
\liminf\limits_{\varepsilon \to 0}\frac{u(x^0-\varepsilon {\bf n}(x^0);t^0)-u(x^0;t^0)}{\varepsilon} >0.
$$
\end{remark}

\begin{remark} \label{remark-10}
Similarly to Remark \ref{remark-5}, one can take as coefficients $b^i$ the signed measures, satisfying condition (\ref{b-par-condition}).
\end{remark}

\section{Some sufficient conditions for the validity of (\ref{b-condition}) and (\ref{b-par-condition})}

In this section we list several simple sufficient conditions on lower-order coefficients providing the validity of assumptions (\ref{b-condition}) and (\ref{b-par-condition}) for elliptic and parabolic operators, respectively. These conditions are close to ones imposed in \cite{Naz12}, where equations in {\it non-divergence form} were studied.

Throughout this section we will denote various constants depending on $n$ only by the letter $N$ without indices, and the constants depending on $n$ and $\gamma$ only by the letter $\widehat{N}$ without indices. 
To simplify the notation, we assume that ${\bf b}$ is extended by zero outside of~$\Omega$ (of $Q$).

In the elliptic case we consider two types of restrictions: distributed drift
\begin{equation} \label{b-1}
\mathbf{b} \in L^n(\Omega); \qquad \sup\limits_{x\in \Omega}\|\mathbf{b}\|_{n, B_{\rho}(x)} \leqslant C \sigma (\rho), 
\end{equation}
and near-boundary drift
\begin{equation} \label{b-2}
|\mathbf{b}(y)| \leqslant C\,\frac{\sigma(d(y))}{d(y)}
\end{equation}
(recall that $\sigma\in \mathcal{D}$).

\begin{lemma} \label{Lemma-A6}
Any of the restrictions (\ref{b-1}) and (\ref{b-2}) implies the validity of condition (\ref{b-condition}). 
\end{lemma}

\begin{proof} 
Let (\ref{b-1}) hold. We set $B^k=B_{r/{2^k}}$ and estimate
$$
\begin{aligned}
\omega(r)&\leqslant\sup\limits_{x\in \Omega}\,\sum\limits_{k=0}^{\infty} \,\,\int\limits_{B^k\setminus B^{k+1}} \frac {|{\bf b}(x-y)|}{|y|^{n-1}}\,dy\\
&\leqslant \sup\limits_{x\in \Omega}\,\sum\limits_{k=0}^{\infty} \,\bigg(\int\limits_{B^k\setminus B^{k+1}}|{\bf b}(x-y)|^n\,dy\bigg)^{\frac 1n}\cdot \bigg(\int\limits_{B^k\setminus B^{k+1}}\frac {dy}{|y|^n}\bigg)^{\frac {n-1}n}\\
&\leqslant NC\sum\limits_{k=0}^{\infty} \sigma(r/2^k)\leqslant NC \mathcal{J}_{\sigma}(2r),
\end{aligned}
$$
and (\ref{b-condition}) follows.

Now 
let (\ref{b-2}) hold. We use the decay of $\sigma (\tau)/\tau$ to estimate
$$
\frac{\sigma(d(y))}{d(y)+|x-y|}\leqslant
\frac{\sigma(d(y)+|x-y|)}{d(y)+|x-y|}\leqslant
\frac{\sigma(|x-y|)}{|x-y|}.
$$
Therefore,
$$
\omega(r)\leqslant C\int\limits_{B_{r}(x)}\frac{\sigma(|x-y|)}{|x-y|^n}\leqslant NC\mathcal{J}_{\sigma}(r),
$$
and (\ref{b-condition}) again follows.
\end{proof}

In the parabolic case we consider the following  analogue of (\ref{b-1}): 
\begin{equation} \label{b-1-par}
\mathbf{b} \in L^{n+1}(Q); \qquad \sup\limits_{(x;t)\in Q} \|\mathbf{b}\|_{n+1, Q_{\rho}(x;t)} \leqslant C \sigma (\rho)\rho^{1/(n+1)}, 
\end{equation}
as well as the analog of (\ref{b-2}):
\begin{equation} \label{b-2-par}
|\mathbf{b}(y;s)| \leqslant C\,\frac{\sigma(d_p(y;s))}{d_p(y;s)}.
\end{equation}

\begin{lemma} \label{Lemma-A7}
Any of the restrictions (\ref{b-1-par}) and  (\ref{b-2-par}) implies the validity of condition (\ref{b-par-condition}) with arbitrary $\gamma>0$.
\end{lemma}

\begin{proof} 
We estimate the quantity $\omega^-_p(r)$; the case of $\omega^+_p(r)$ is considered along the same lines.

Let (\ref{b-1-par}) hold. We set $Q^k=Q_{r/{2^k}}$ and obtain
$$
\begin{aligned}
\omega^-_p(r)&\leqslant\sup\limits_{(x;t)\in Q}\, \sum\limits_{k=0}^{\infty}\, \int\limits_{Q^k\setminus Q^{k+1}}\frac{|\mathbf{b}(x-y;t+s)|}{(-s)^{(n+1)/2}}\cdot \exp{\left(-\gamma\,\frac{|y|^2}{-s}\right)}\,dy ds
\\
&\leqslant \sup\limits_{(x;t)\in Q}\,\sum\limits_{k=0}^{\infty} \,\bigg(\int\limits_{Q^k\setminus Q^{k+1}}
|{\bf b}(x-y;t+s)|^{n+1}\,dyds\bigg)^{1/(n+1)}
\cdot \Phi_k^{n/(n+1)}\\
&\leqslant C\sum\limits_{k=0}^{\infty} \sigma(r/2^k)
\cdot(r/2^k)^{1/(n+1)}\cdot \Phi_k^{n/(n+1)},
\end{aligned}
$$
where
$$
\Phi_k=\int\limits_{Q^k\setminus Q^{k+1}}
\exp{\left(-\gamma\,\frac{n+1} n\,\frac{|y|^2}{-s}\right)}\cdot
\frac{dyds}{(-s)^{(n+1)^2/2n}}.
$$

Change of variables $\varrho=|y|/\sqrt{-s}$, 
$\tau=\sqrt{|y|^2-s}$ gives
$$
\Phi_k\le\int\limits_{r/2^{k+1}}^{r\sqrt{2}/2^k}
\int\limits_0^{\infty}
\exp{\left(-\gamma\,\frac{n+1} n\,\varrho^2\right)}\cdot
\frac{2\varrho^{n-1}(\varrho^2+1)^{1/2n}}{\tau^{1+1/n}}\,d\varrho d\tau\leqslant \frac {\widehat N}{(r/2^k)^{1/n}}.
$$
Thus,
$$
\omega^-_p(r)\leqslant \widehat NC\sum\limits_{k=0}^{\infty} \sigma(r/2^k)
\leqslant \widehat NC \mathcal{J}_{\sigma}(2r),
$$
and the first relation in (\ref{b-par-condition}) follows.

Now 
let (\ref{b-2-par}) hold. As in Lemma \ref{Lemma-A6}, we use the decay of $\sigma (\tau)/\tau$ to estimate
$$
\frac{\sigma(d_p(y;s))}{d_p(y;s)+\sqrt{|x-y|^2+t-s}}\leqslant
\frac{\sigma(\sqrt{|x-y|^2+t-s})}{\sqrt{|x-y|^2+t-s}}.
$$
Therefore,
$$
\omega^-_p(r)\leqslant C\int\limits_{Q_{r}(x;t)}\exp{\left(-\gamma\,\frac{|x-y|^2}{t-s}\right)}\cdot\frac{\sigma(\sqrt{|x-y|^2+t-s})}{\sqrt{|x-y|^2+t-s}}\,\frac{dyds}{(t-s)^{(n+1)/2}}.
$$
Change of variables $\varrho=|x-y|/\sqrt{t-s}$, 
$\tau=\sqrt{t-s+|x-y|^2}$ gives
$$
\omega^-_p(r)\leqslant C\int\limits_0^{r\sqrt{2}}
\int\limits_0^{\infty}\exp{\left(-\gamma\varrho^2\right)}\cdot\frac {\varrho^{n-1}}{\sqrt{\varrho^2+1}} \,\frac{\sigma (\tau)}{\tau}\, d\varrho d\tau\leqslant \widehat NC\mathcal{J}_{\sigma}(r\sqrt{2}),
$$
and the first relation in (\ref{b-par-condition}) again follows.
\end{proof}

\bibliography{Bibliography_HopfDiv_}


\Addresses

\end{document}